\documentclass[12pt,a4paper]{article}
\usepackage[utf8]{inputenc}
\usepackage[english]{babel}
\usepackage{lipsum}
\usepackage[T1]{fontenc}
\usepackage{amsmath}
\usepackage{amsthm}
\usepackage{amsfonts}
\usepackage{amssymb}
\usepackage{graphicx}
\usepackage{hyperref}
\usepackage{enumerate}
\usepackage{enumitem}
\usepackage{subfigure}
\usepackage{color}
\usepackage{xcolor}
\usepackage{authblk}
\allowdisplaybreaks

\usepackage{combelow}

\usepackage{newunicodechar}
\newunicodechar{ș}{\cb{s}}
\newunicodechar{ț}{\cb{t}}

\newtheorem{theorem}{Theorem}[section]

\newtheorem{lemma}{Lemma}[section]

\DeclareMathOperator{\argmin}{arg\,min}

\begin{document}

\title{Nonlocal Stochastic Optimal Control for Diffusion Processes: Existence, Maximum Principle and Financial Applications 
}

\author[1]{ \c{S}tefana-Lucia Ani\c{t}a}
\author[2]{Luca Di Persio}
\affil[1]{Octav Mayer Institute of Mathematics of the Romanian Academy,
Bd. Carol I 8, Ia\c{s}i 700505, Romania, stefi$\_$anita@yahoo.com}
\affil[2]{College of Mathematics - Computer Science Department - Verona University, Str. le Grazie, 15, 37134 Verona, Italy, luca.dipersio@univr.it}


\date{}

\maketitle


\begin{abstract}
This paper investigates the optimal control problem for a class of parabolic equations where the diffusion coefficient is influenced by a control function acting nonlocally. Specifically, we consider the optimization of a cost functional that incorporates a controlled probability density evolving under a Fokker-Planck equation with state-dependent drift and diffusion terms. The control variable is subject to spatial convolution through a kernel, inducing nonlocal interactions in both drift and diffusion terms. We establish the existence of optimal controls under appropriate convexity and regularity conditions, leveraging compactness arguments in function spaces. A maximum principle is derived to characterize the optimal control explicitly, revealing its dependence on the adjoint state and the nonlocal structure of the system. We further provide a rigorous financial application in the context of mean-variance portfolio optimization, where both the asset drift and volatility are controlled nonlocally, leading to an integral representation of the optimal investment strategy. The results offer a mathematically rigorous framework for optimizing diffusion-driven systems with spatially distributed control effects, broadening the applicability of nonlocal control methods to stochastic optimization and financial engineering.
\end{abstract}


\paragraph{MSC:\\}
{
\noindent
49J20; 49K45; 60H10; 60J60; 91G80; 93E20
}

 \bigskip
 
 
\section{Introduction}
\label{sec:intro}

In some applications, it is necessary to control the solution to a parabolic equation by acting on the diffusion
coefficient via a control with a nonlocal action. Such a problem appears, for example, when we wish to 
optimally displace a population via a control that decreases the diffusion coefficient by increasing the food resources 
or augmenting the diffusion coefficient by reducing the food resources. An application in Mathematical Finance will be discussed in Section 4.

Actually, we deal with the following optimal control problem
$$Minimize \  \int_0^T\int_{\mathbb{R}^d} G(t,x)\rho ^u(t,x) \, dx \, dt +\int_{\mathbb{R}^d}G_T(x)\rho ^u(T,x) \, dx +\int_{\mathbb{R}^d}h(u(x)) \, dx \, \leqno{\bf (P)}$$
where 
$$u\in \{ v\in L^2(\omega ) \ : \ v(x)\in [0,M_0] \ \mbox{\rm a.e. } x\in \omega \} ={\cal U} \, , $$
$M_0\in (0,+\infty )$, $\omega \subset \mathbb{R}^d$ ($d\in \mathbb{N}^*$) is an open and bounded set and $\rho ^u$ is the unique weak solution to the following Fokker-Planck equation
\begin{equation}
\left\{ \begin{array}{ll}
\displaystyle
{{\partial \rho }\over {\partial t}}(t,x)=-\nabla \cdot (b(t,x, (Su)(x))\rho (t,x)) \qquad & \\
\hspace{10mm} \displaystyle +{1\over 2}{{\partial ^2}\over {\partial x_i\partial x_j}}\Big( q_{ij}(t,x,(Su)(x))\rho (t,x)\Big), &t\in (0,T), \ x\in \mathbb{R}^d ,\vspace{2mm}\\
\rho(0,x)=\rho_0(x), & x\in \mathbb{R}^d \, .
\end{array} \right. \label{FP1}
\end{equation}
Here, $T\in (0,+\infty )$ and for any $u\in L^2(\omega )$ we denote by the same symbol $u$, the extension by $0$ in $\mathbb{R}^d\setminus \omega $. In the paper we use the Einstein summation convention.
\vspace{2mm}

\noindent
\bf Definition. \rm We say that $\rho \in L^2(0,T;H^1(\mathbb{R}^d))\cap W^{1,2}([0,T];H^{-1}(\mathbb{R}^d))$ is a weak solution to (\ref{FP1})
if
$$\begin{array}{ll}
\displaystyle
\Big\langle \frac{d\rho }{dt}(t),\varphi \Big\rangle _{H^{-1},H^1}=& \displaystyle \int_{\mathbb{R}^d}b(t,x,(Su)(x))\rho (t,x)\cdot \nabla \varphi (x)dx \vspace{2mm} \\
~ & \displaystyle -\frac{1}{2}\int_{\mathbb{R}^d}\frac{\partial }{\partial x_j}(q_{ij}(t,x,(Su)(x))\rho (t,x))\frac{\partial \varphi }{\partial x_i}(x)dx 
\end{array}$$
a.e. $t\in (0,T), \ \forall \varphi \in H^1(\mathbb{R}^d)$

\noindent
and 
$$\rho (0,x)=\rho _0(x) \quad \mbox{\rm a.e. } x\in \mathbb{R}^d \, .$$ 
\vspace{2mm}

Here $\langle \cdot ,\cdot \rangle_{H^{-1},H^1}$ is the duality between $H^{-1}(\mathbb{R}^d)$ and $H^1(\mathbb{R}^d)$.
We shall use sometimes the simplified notations $H^1$ and $H^{-1}$ for $H^1(\mathbb{R}^d)$ and $H^{-1}(\mathbb{R}^d)$, respectively and $L^{\tilde{q}}$ instead of $L^{\tilde{q}}(A)$ when is no danger of confusion. 

Recall that $L^2(0,T;H^1(\mathbb{R}^d))\cap W^{1,2}([0,T];H^{-1}(\mathbb{R}^d))\subset C([0,T];L^2(\mathbb{R}^d))$ (see \cite{barbu}).
\vspace{2mm}

For any $u\in L^2(\omega )$ and define 
$$(Su)(x)=\int_{\mathbb{R}^d}K(x-x')u(x') \, dx' \, . $$
The kernel $K$ satisfies the following assumptions 
$$K\in W^{1,\ell }(\mathbb{R}^d), \quad K(x)\geq 0, \ \forall x\in \mathbb{R}^d \, ,$$
where $\ell \in [1,2]$.
\vspace{2mm}

Let us notice first that 
\begin{itemize}
\item{} For any $u\in L^2(\omega )$ we have via Young's inequality for convolutions that $Su \in L^r(\mathbb{R}^d)$ and
$$\| Su\|_{L^r}\leq \| K\|_{L^{\ell }}\| u\|_{L^2(\omega )} \, ,$$
where $\frac{1}{r}+1=\frac{1}{\ell }+\frac{1}{2}$ (which implies that $r=\frac{2\ell }{2-\ell }\in [2,+\infty ]$).
\end{itemize}
\vspace{1mm}

On the other hand, by Lemma VIII.4, p.128 in \cite{brezis}, we get that 
\begin{itemize}
\item{} For any $u\in L^2(\omega )$:
$$\nabla (Su)(x)=\int_{\mathbb{R}^d}\nabla K(x-x')u(x') \, dx' $$
(the gradient is in the sense of distribution theory) and satisfies
$$\| \nabla (S u)\|_{L^r}\leq \| \nabla K\| _{L^{\ell }}\| u\| _{L^2(\omega )} \, .$$
\end{itemize}
We may conclude that $S\in L(L^2(\omega );W^{1,r}(\mathbb{R}^d))$.
\vspace{3mm}

\noindent
\bf Remark. \rm Let us notice that for any $u\in {\cal U}$ it follows that $u\in L^{{\ell }^*}(\omega )$, where $\frac{1}{\ell }+\frac{1}{\ell ^*}=1$ and that
by the above mentioned convention about extension outside $\omega $, that $u\in L^{\ell ^*}(\mathbb{R}^d)$. Note also that
$$\|u\| _{L^{\ell ^*}}\leq M_0 \, m(\omega )^{\frac{1}{\ell ^*}} \, ,$$
where we denote by $m(A)$ the Lebesgue measure of $A$.

By Young's inequality for convolutions, we get that for any $u\in {\cal U}$:
$$\| Su\| _{L^{\infty }}\leq \| K\| _{L^{\ell }}\| u\| _{L^{\ell ^*}}\leq M_0 \, m(\omega )^{\frac{1}{\ell ^*}}\| K\| _{L^{\ell }} \, ,$$
$$\| \nabla (Su)\| _{L^{\infty }}\leq \| \nabla K\| _{L^{\ell }}\| u\| _{L^{\ell ^*}}\leq M_0 \, m(\omega )^{\frac{1}{\ell ^*}}\| \nabla K\| _{L^{\ell }} \, .$$
\vspace{5mm}

In the present paper we work under the following hypotheses

\begin{itemize}
\item[\bf (H1)] $b\in L^{\infty }((0,T)\times \mathbb{R}^d\times [0,+\infty );\mathbb{R}^d)$, $b$ is a Borel function and for almost any $(t,x)\in (0,T)\times \mathbb{R}^d$, $b(t,x,\cdot )\in C([0,+\infty ))$;
\item[\bf (H2)] $\sigma \in C_b^2([0,T]\times \mathbb{R}^d\times [0,+\infty );\mathbb{R}^{d\times N})$, $N\in \mathbb{N}^*$,
and that there exists a positive constant $\gamma $ such that
$$\sigma _{il}(t,x,s)\sigma_{jl}(t,x,s)\xi _i\xi_j\geq \gamma |\xi |^2_d, 
\quad \forall t\in [0,T], \ x\in \mathbb{R}^d, \ s\geq 0 \, ;$$
\item[\bf (H3)] $\rho _0\in L^2(\mathbb{R}^d)\cap L^1(\mathbb{R}^d)$, $\rho _0(x)\geq 0$ a.e. $x\in \mathbb{R}^d$, $\displaystyle \int_{\mathbb{R}^d}\rho _0(x) dx = 1$;
\item[\bf (H4)] $G\in L^2((0,T)\times \mathbb{R}^d), \ G_T\in L^2(\mathbb{R}^d)$,
$$G(t,x)\geq 0 \ \mbox{\rm a.e. } (t,x)\in (0,T)\times \mathbb{R}^d, \quad G_T(x)\geq 0 \ \mbox{\rm a.e. } x\in \mathbb{R}^d \, ,$$
$h:[0,M_0]\longrightarrow [0,+\infty )$ is convex, continuous and there exists a positive constant $\alpha$ such that
$$h(s)\geq \frac{\alpha }{2}|s|^2, \quad \forall s\geq 0 \, .$$
\end{itemize}

Let 
$$q_{ij}=\sigma _{il}\sigma_{jl}\in C_b^2([0,T]\times \mathbb{R}^d\times [0,+\infty )),$$
for any $i,j\in \{ 1,2,...,d\} $. 
\vspace{2mm} 

For any $u\in {\cal U}$, it follows via Lions' theorem (see \cite{brezis}, p.218) that there exists a unique weak solution $\rho ^u$ to (\ref{FP1}).
As in \cite{anita} it follows that $\rho ^u$ is also a distributional solution to (\ref{FP1}), that $\rho ^u\in C([0,T];L^1(\mathbb{R}^d))$ (and so, it is t-narrowly continuous) and that $\rho ^u(t)$ is a probability density for any $t\in [0,T]$. By the superposition principle (see \cite{figalli,HRW,trevisan}) we get that there exists a unique (in law) probabilistically
weak solution $(X^u(t))_{t\in [0,T]}$ to  
\begin{equation}
\left\{ \begin{array}{ll}
\displaystyle
dX(t)=b(t,X(t), (Su)(X(t)))dt \vspace{2mm} \\
\hspace{3cm}+ \sigma (t,X(t), (Su)(X(t)))dW(t), \quad &t\in (0,T), \vspace{2mm}\\
X(0)=X_0,  &  
\end{array} \right. \label{SDE}
\end{equation}
such that $X^u(t)$ has the probability density $\rho ^u(t)$ for any $t\in [0,T]$. Actually, a weak solution $(X^u(t))_{t\in [0,T]}$
to (\ref{SDE}) (see \cite{oksendal,evans}) is a strong solution to (\ref{SDE}) corresponding to a probability space $({\Omega },{\cal F}, {\mathbb{P}}):=(\tilde{\Omega },\tilde{\cal F}, \tilde{\mathbb{P}})$ with normal filtration $(\tilde{\cal F}_t)_{t\in [0,T]}$ and to $({W}(t))_{t\in [0,T]}:=(\tilde{W}(t))_{t\in [0,T]}$, an $\mathbb{R}^N-(\tilde{\cal F}_t)_{t\in [0,T]}$ Brownian motion. 

So, for any $u\in {\cal U}$:
$$\int_0^T\mathbb{E}[G(t,X^u(t))] \, dt +\mathbb{E}[G_T(X^u(T))] =\int_0^T\int_{\mathbb{R}^d}G(t,x)\rho ^u(t,x)dx \, dt$$
$$+\int_{\mathbb{R}^d}G_T(x)\rho ^u(T,x)dx . $$
It follows that problem $(P)$ is equivalent to the following stochastic optimal control problem 
$$\underset{u\in {\cal U}}{Minimize} \  \int_0^T\mathbb{E}[G(t,X^u(t))] \, dt +\mathbb{E}[G_T(X^u(T))] +\int_{\mathbb{R}^d}h(u(x)) \, dx \, .\leqno{\bf (P_S)}$$
\vspace{3mm}

Consider now the stochastic differential equation (\ref{SDE}), where $({\Omega },{\cal F}, {\mathbb{P}})$ is a probability space with normal filtration $({\cal F}_t)_{t\in [0,T]}$ and $({W}(t))_{t\in [0,T]}$ is an $\mathbb{R}^N-({\cal F}_t)_{t\in [0,T]}$ Brownian motion.  

Assume that $X_0$ is an $\mathbb{R}^d$-random variable with density $\rho _0$ and that $b\in C^1_b([0,T]\times \mathbb{R}^d\times [0,+\infty );\mathbb{R}^d)$ and $K\in C_0^1(\mathbb{R}^d)$. For any $u\in {\cal U}$, problem (\ref{SDE}) admits a unique strong solution $(X^u(t))_{t\in [0,T]}$.
On the other hand (\ref{FP1}) admits a unique weak (and distributional as well) solution $\rho ^u$ (with $\rho ^u(t)$ a probability density for any $t\in [0,T]$). By Theorem 1.3 in \cite{rockner} we get that $\rho ^u(t)$ is the probability density for $X^u(t)$ for any $t\in [0,T]$. We conclude that $(P)$ and $(P_S)$ are equivalent, this time $X^u$ being the strong solution to (\ref{SDE}).  
\vspace{3mm}

We shall investigate, in what follows, problem $(P)$ under the general assumptions on $b$ and $K$.
\vspace{1cm}

The study of optimal control problems for diffusion processes where the control acts on the diffusion coefficient has been a longstanding challenge in stochastic control theory, mathematical finance, and statistical mechanics. Early foundational work on controlled diffusions, particularly in the context of the stochastic maximum principle (SMP) and dynamic programming, can be traced to Bensoussan (\cite{bensou}), Fleming and Soner (\cite{fleming}), and Yong and Zhou (\cite{yong}). These references established the key analytic and probabilistic tools necessary for tackling optimization in diffusion-driven systems, primarily under the assumption of control acting in the drift.

The study of diffusion control, however, requires more sophisticated techniques due to the degenerate nature of the associated Hamilton-Jacobi-Bellman (HJB) equations and the resulting forward-backward stochastic differential equations (FBSDEs), as well as the maximum principle in the presence of jumps and w.r.t. mean field games theory, see, e.g., \cite{li} and the seminal contributions of Øksendal and Sulem (\cite{oksendal2}), where diffusion control was systematically analyzed via FBSDEs and viscosity solutions of second-order HJB equations. More recent advances have been obtained in \cite{zhang}, where the authors study optimal stochastic control problems for fully coupled FBSDEs subjected to
anomalous subdiffusion, then obtaining the stochastic maximum principle (SMP). 

In the context of controlled Fokker-Planck equations, where the probability distribution of a stochastic system is directly controlled via drift or diffusion, recent literature has extensively developed mean-field control (MFC) and McKean-Vlasov (MKV) control problems. Lasry and Lions (\cite{lasry}) introduced the fundamental analysis of mean-field games, where control influences a distributional state variable, which naturally extends to problems where the diffusion coefficient is controlled nonlocally. Carmona and Delarue (\cite{carmona}) developed an extensive probabilistic framework for MFC problems, incorporating diffusion control, stochastic maximum principles, and HJB-Fokker-Planck systems.

The nonlocal nature of control, where diffusion depends on a spatially weighted convolution of the control function, introduces additional complexity. Such settings are relevant to study maximum entropy approximation, see \cite{bodova}, as well as for the {\it consensus-based} optimisation problems, see, e.g., \cite{HQR}, and mean-field game theory (Bensoussan, Frehse, and Yam (\cite{bensou2})). Rigorous results concerning the well-posedness of Fokker-Planck equations with nonlocal diffusion terms have been obtained by Jourdain, M\'el\'eard, and Woyczynski (\cite{jourdain}), who analyzed solutions in non-standard function spaces, ensuring stability of the density evolution.

A crucial mathematical challenge in diffusion control is establishing existence and uniqueness of optimal controls under nonlocal interactions. Here, works by Buckdahn, Li, Peng, and Rainer (\cite{buckdahn}) provide deep insights into adjoint BSDEs for stochastic control of McKean-Vlasov dynamics. The concept of mean-field optimal control introduced in \cite{fornasier} has proven effective in describing controlled aggregation-diffusion systems, a problem structurally similar to the one considered in this paper.

Finally, financial applications of diffusion control, particularly in mean-variance portfolio optimization with stochastic volatility, have been extensively studied in the works of Kraft (\cite{kraft}) and Fouque, Papanicolaou, and Sircar (\cite{fouque}). These contributions underline the importance of nonlocal control in risk-sensitive portfolio allocations, where the volatility structure of the market evolves under an optimized diffusion process.
\vspace{5mm}

Here is the structure of the paper. In section 2 we prove the existence of an optimal control for problem $(P)$. In section 3 we derive the maximum principle under more restrictive assumptions. An application in Mathematical Finance is discussed in section 4. Section 5 is devoted to some comments and further developments.

\section{Existence of an optimal control}

For any $u\in {\cal U}$ we denote by $J(u)$ the cost functional at $u$.

\begin{lemma}
If $\{ u_n\}_{n\in \mathbb{N}^*}\subset {\cal U}$ satisfies $u_n\longrightarrow u$ weakly  in $L^{\ell ^*}(\omega )$, then
$$(Su_n)(x)\longrightarrow (Su)(x), \quad \frac{\partial (Su_n)}{\partial x_i}(x)\longrightarrow \frac{\partial (Su)}{\partial x_i}(x) \quad a.e. \ x\in \mathbb{R}^d \, ,$$
for any $i\in \{1,2, ...,d\} $.
\end{lemma}

\begin{proof}
Note first that $u\in {\cal U}$. 

Since $u_n\longrightarrow u$ weakly in $L^{\ell ^*}(\mathbb{R}^d)$, as well (where the functions in $L^{\ell ^*}(\omega )$ are extended by $0$ on $\mathbb{R}^d\setminus \omega $ and thus become elements of $L^{\ell ^*}(\mathbb{R}^d)$), then for almost any $x\in \mathbb{R}$
we get that
$$(Su_n)(x)-(Su)(x)=\int_{\mathbb{R}^d}K(x-x')(u_n(x')-u(x')) \, dx' \longrightarrow 0 $$
(we get here the duality product of $K(x-\cdot )$ and $u_n-u$) and analogously, we get that
for almost any $x\in \mathbb{R}^d$ and $i\in \{ 1,2,...,d\} $:
$$\frac{\partial }{\partial x_i}(Su_n-Su)(x)=\int_{\mathbb{R}^d}\frac{\partial K}{\partial x_i}(x-x')(u_n(x')-u(x')) \, dx' \longrightarrow 0 \, .$$
It follows the conclusion.
\end{proof}

\begin{theorem}
There exists at least one optimal control of problem $(P)$.
\end{theorem}

\begin{proof}
Let
$$m=\inf _{u\in {\cal U}}J(u)\in [0,+\infty ) \, $$
and $\{ u_n\} _{n\in \mathbb{N}^*}\subset {\cal U}$ such that
$$m\leq J(u_n)<m+\frac{1}{n}, \quad \forall n\in \mathbb{N}^* \, .$$
It follows that there exists a subsequence $\{ u_{n_k}\}_{k\in \mathbb{N}^*}$ and $u^*\in {\cal U}$ such that
$$u_{n_k}\longrightarrow u^*  \quad \mbox{\rm weakly in } L^2(\mathbb{R}^d) \ \mbox{\rm and  weakly }  \ \mbox{\rm in } L^{\ell ^*}(\mathbb{R}^d) \, .$$
Let us prove that
$$\rho^{u_{n_k}}\longrightarrow \rho^{u^*} \quad \mbox{\rm in } C([0,T]; L^2(\mathbb{R}^d)) $$
(where $\rho ^u$ is the weak solution to (\ref{FP1}) in the sense of the Theorem of Lions - see \cite{brezis}).
Indeed, if we denote by $z_k=\rho^{u_{n_k}}-\rho^{u^*}$, then $z_k$ is the weak solution to
\begin{equation*}
\left\{ \begin{array}{ll}
\displaystyle
\frac{\partial z }{\partial t}(t,x)=-\nabla \cdot \Big[ b(t,x,(Su_{n_k})(x))\rho ^{u_{n_k}}(t,x)-b(t,x,(Su^*)(x))\rho ^{u^*}(t,x)\Big]  \\
\hspace{5mm} \displaystyle +\frac{1}{2}\frac{\partial ^2}{\partial x_i\partial x_j}\Big[ q_{ij}(t,x,(Su_{n_k})(x))\rho^{u_{n_k}}(t,x)-q_{ij}(t,x,(Su^*)(x))\rho^{u^*}(t,x)\Big], \\
~ \hspace{5cm} t\in (0,T), \ x\in \mathbb{R}^d ,\vspace{2mm}\\
z(0,x)=0, \hspace{3cm} x\in \mathbb{R}^d 
\end{array} \right. 
\end{equation*}
and consequently also to
\begin{equation}
\left\{ \begin{array}{ll}
\displaystyle
\frac{\partial z }{\partial t}(t,x)=-\nabla \cdot (b(t,x,(Su_{n_k})(x))z (t,x)) \vspace{2mm} \\
\hspace{5mm} -\nabla \cdot (b(t,x,(Su_{n_k})(x))-b(t,x,(Su^*)(x)))\rho^{u^*}(t,x) \vspace{2mm} \\
\hspace{5mm} \displaystyle +\frac{1}{2}\frac{\partial ^2}{\partial x_i\partial x_j} \Big[ q_{ij}(t,x,(Su_{n_k})(x))z(t,x)\Big],  \vspace{2mm} \\
\hspace{5mm} \displaystyle +\frac{1}{2}\frac{\partial ^2}{\partial x_i\partial x_j}\Big[ \Big( q_{ij}(t,x,(Su_{n_k})(x))-q_{ij}(t,x,(Su^*)(x))\Big) \rho^{u^*}(t,x) \Big], \vspace{2mm} \\
~ \hspace{5cm} t\in (0,T), \ x\in \mathbb{R}^d ,\vspace{2mm}\\
z(0,x)=0, \hspace{3cm} x\in \mathbb{R}^d \, .
\end{array} \right. \label{diff_FP1}
\end{equation}
For the sake of avoiding writing excessively long formulas, we shall sometimes shorthand $q_{ij}(\tau ,x,s)$ by $q_{ij}(s )$ and $b(\tau ,x,s)$ by $b(s)$.
Since $z_k$ is the weak solution to (\ref{diff_FP1}) and using Proposition IX.5, p.155 in \cite{brezis}, we get that for any $t\in [0,T]$:
\begin{equation*}
\begin{array}{ll}
~ &\displaystyle
\hspace{3cm} \frac{1}{2}\| z_k(t)\|^2_{L^2} \vspace{2mm} \\
 = & \displaystyle  \int_0^t\int_{\mathbb{R}^d}b(\tau ,x,(Su_{n_k})(x))z_k(\tau ,x)\cdot \nabla z_k(\tau ,x) \, dx \, d\tau   \vspace{2mm} \\
 ~ & \displaystyle +\int_0^t\int_{\mathbb{R}^d}
 \Big[b(\tau ,x,(Su_{n_k})(x))-b(\tau ,x,(Su^*)(x)) \Big] \cdot \nabla z_k(\tau ,x)\rho^{u^*}(\tau ,x) \, dx \, d\tau \vspace{2mm} \\ 
 \ \ \ & \displaystyle -\frac{1}{2}\int_0^t\int_{\mathbb{R}^d}\frac{\partial z_k}{\partial x_i}(\tau ,x)\frac{\partial }{\partial x_j}[q_{ij}(\tau ,x,(Su_{n_k})(x))z_{k}(\tau ,x)] \, dx \, d\tau  \vspace{2mm} \\
 \ \ \ & \displaystyle -\frac{1}{2}\int_0^t\int_{\mathbb{R}^d}\frac{\partial z_k}{\partial x_i}(\tau ,x)\frac{\partial }{\partial x_j}[(q_{ij}(\tau ,x,(Su_{n_k})(x))-q_{ij}(\tau ,x,(Su^*)(x)))\rho^{u^*}(\tau ,x)] \, dx \, d\tau  \vspace{2mm} \\
 \displaystyle
 \leq & \displaystyle \| b\| _{L^{\infty }}\int_0^t\| z_k(\tau )\|_{L^2}\| \nabla z_k(\tau )\|_{L^2} \, d\tau  \vspace{2mm} \\
 \ \ \ & \displaystyle +\int_0^t\int_{\mathbb{R}^d}
 |b(\tau ,x,(Su_{n_k})(x))-b(\tau ,x,(Su^*)(x))| \, |\nabla z_k(\tau ,x)|\rho^{u^*}(\tau ,x) \, dx \, d\tau \vspace{2mm} \\
 \ \ \ & \displaystyle -\frac{\gamma }{2}\int_0^t \|\nabla z_{k}(\tau )\| ^2_{L^2} \, d\tau  \displaystyle -\frac{1}{2}\int_0^t\int_{\mathbb{R}^d}\frac{\partial z_k}{\partial x_i}(\tau ,x)z_k(\tau ,x)\Big[ \frac{\partial q_{ij}}{\partial x_j}(\tau ,x ,(Su_{n_k})(x)) \vspace{2mm} \\
~ & \displaystyle \hspace{3cm} +\frac{\partial q_{ij}}{\partial s}(\tau ,x,(Su_{n_k})(x))\frac{\partial (Su_{n_k})}{\partial x_j}(x)\Big] \, dx \, d\tau \vspace{2mm} \\
 & \displaystyle -\frac{1}{2}\int_0^t\int_{\mathbb{R}^d}\frac{\partial z_k}{\partial x_i}(\tau ,x)\Big\{ \frac{\partial \rho^{u^*}}{\partial x_j}(\tau ,x)[(q_{ij}(\tau ,x,(Su_{n_k})(x))-q_{ij}(\tau ,x, (Su^*)(x))] \vspace{2mm} \\
 & \displaystyle  \ \ \ + \rho^{u^*}(\tau ,x)\Big[ \frac{\partial q_{ij}}{\partial x_j}(\tau ,x,(Su_{n_k})(x))- \frac{\partial q_{ij}}{\partial x_j}(\tau ,x,(Su^*)(x)) \Big]  \vspace{2mm} \\
 & \displaystyle  \ \ \ + \rho^{u^*}(\tau ,x)\Big[ \frac{\partial q_{ij}}{\partial s}(\tau ,x,(Su_{n_k})(x))\frac{\partial (Su_{n_k})}{\partial x_j}(x) \vspace{2mm} \\
~ & \displaystyle \hspace{3cm}- \frac{\partial q_{ij}}{\partial s}(\tau ,x,(Su^*)(x))\frac{\partial (Su^*)}{\partial x_j}(x)\Big] \Big\} \, dx \, d\tau \, .
\end{array} 
\end{equation*}
After an easy calculation, we get that there exist two positive constants $\beta, \theta $ such that for any $k\in \mathbb{N}^*$:
\begin{equation*}
\begin{array}{ll}
\displaystyle
\frac{1}{2}\| z_k(t)\|^2_{L^2}+\frac{\gamma }{4}\int_0^t\| \nabla z_k(\tau )\| ^2_{L^2}d\tau \leq \frac{\beta }{2} \int_0^t\| z_k(\tau )\| ^2_{L^2}d\tau \vspace{2mm} \\
\displaystyle +\frac{\theta }{2} \int_0^T\int_{\mathbb{R}^d}\Big[ \Big| \frac{\partial \rho^{u^*}}{\partial x_j}(\tau ,x)\Big|^2|q_{ij}(\tau ,x,(Su_{n_k})(x))-q_{ij}(\tau ,x,(Su^*)(x))|^2 \vspace{2mm} \\
\displaystyle +|\rho^{u^*}(\tau ,x)|^2\Big| \frac{\partial q_{ij}}{\partial x_j}(\tau ,x,(Su_{n_k})(x))-\frac{\partial q_{ij}}{\partial x_j}(\tau ,x,(Su^*)(x))\Big| ^2 \vspace{2mm} \\
\displaystyle +|\rho^{u^*}(\tau ,x)|^2\Big| \frac{\partial q_{ij}}{\partial s}(\tau ,x,(Su_{n_k})(x))\frac{\partial (Su_{n_k})}{\partial x_j}(x)-\frac{\partial q_{ij}}{\partial s}(\tau ,x,(Su^*)(x))\frac{\partial (Su^*)}{\partial x_j}(x)\Big| ^2 \vspace{2mm} \\
\displaystyle +|\rho^{u^*}(\tau ,x)|^2|b(\tau ,x,(Su_{n_k})(x)) -b(\tau ,x,(Su^*)(x)) |^2 \Big] \, dx \, d\tau \, .
\end{array} 
\end{equation*}
If we denote by 
$$F_k(\tau ,x)= \Big| \frac{\partial \rho^{u^*}}{\partial x_j}(\tau ,x)\Big|^2|q_{ij}(\tau ,x,(Su_{n_k})(x))-q_{ij}(\tau ,x,(Su^*)(x))|^2 $$
$$+|\rho^{u^*}(\tau ,x)|^2\Big| \frac{\partial q_{ij}}{\partial x_j}(\tau ,x,(Su_{n_k})(x))-\frac{\partial q_{ij}}{\partial x_j}(\tau ,x,(Su^*)(x))\Big| ^2 $$
$$+|\rho^{u^*}(\tau ,x)|^2\Big| \frac{\partial q_{ij}}{\partial s}(\tau ,x,(Su_{n_k})(x))\frac{\partial (Su_{n_k})}{\partial x_j}(x)-\frac{\partial q_{ij}}{\partial s}(\tau ,x,(Su^*)(x))\frac{\partial (Su^*)}{\partial x_j}(x)\Big| ^2 $$
$$+|\rho^{u^*}(\tau ,x)|^2|b(\tau ,x,(Su_{n_k})(x)) -b(\tau ,x,(Su^*)(x)) |^2 \, ,$$
then we may infer (via Gronwall's lemma) that
$$\| z_k(t)\|^2_{L^2}\leq e^{\beta T}\theta \int_0^T\int_{\mathbb{R}^d}F_k(\tau ,x) \, dx \, d\tau , \quad \mbox{\rm for any } t\in [0,T] \, .$$
Note that  $b(Su_{n_k})-b(Su^*)$, $q_{ij}(Su_{n_k})-q_{ij}(Su^*)$, $\frac{\partial q_{ij}}{\partial x_j}(Su_{n_k})-\frac{\partial q_{ij}}{\partial x_j}(Su^*)$ and
$\frac{\partial q_{ij}}{\partial s}(Su_{n_k})\frac{\partial (Su_{n_k})}{\partial x_j}-\frac{\partial q_{ij}}{\partial s}(Su^*)\frac{\partial (Su^*)}{\partial x_j}$ are bounded and 
almost everywhere convergent to $0$.

So, by Lebesgue's dominated convergence theorem, we may conclude that
$$\int_0^T\int_{\mathbb{R}^d}F_k(\tau ,x) \, dx \, d\tau \longrightarrow 0 \, ,$$
and so
$$z_k\longrightarrow  0 \quad \mbox{\rm in } C([0,T];L^2(\mathbb{R}^d))$$
and 
$$\rho ^{u_{n_k}}\longrightarrow  \rho^{u^*} \quad \mbox{\rm in } C([0,T];L^2(\mathbb{R}^d)) .$$
The last convergence implies that
$$\int_0^T\int_{\mathbb{R}^d}G(t,x)\rho ^{u_{n_k}}(t,x)dx \, dt+\int_{\mathbb{R}^d}G_T(x)\rho ^{u_{n_k}}(T,x)dx $$
$$\longrightarrow  \int_0^T\int_{\mathbb{R}^d}G(t,x)\rho^{u^*}(t,x)dx \, dt+\int_{\mathbb{R}^d}G_T(x)\rho^{u^*}(T,x)dx \, .$$
On the other hand, the assumptions on $h$ allow us to conclude that
$$\liminf_{k\rightarrow +\infty }\int_{\mathbb{R}^d}h(u_{n_k}(x)) \, dx \geq \int_{\mathbb{R}^d}h(u^*(x)) \, dx$$
and consequently that 
$$m\geq \liminf_{k\rightarrow +\infty } J(u_{n_k})\geq J(u^*)\geq m,$$
which means that $u^*$ is an optimal control for $(P)$.
\end{proof}

\section{The maximum principle}

Assume in what follows that besides the previously considered hypotheses, we have that
\begin{itemize}
\item[\bf (H5)] $b$, $\sigma $ are time-independent, $b\in C^{0,2}_b(\mathbb{R}^d\times [0,+\infty);\mathbb{R}^d)$ and
$$\sigma _{ij}(x,s)=\tilde{\sigma }(x,s)\delta _{ij}, \quad \forall i,j\in \{ 1,2,...,d\} ,$$
where $\delta _{ij}$ is the Kronecker delta;
$$G_T\in H^1(\mathbb{R}^d) \, .$$
\end{itemize}

Let us denote by $q(x,s)=\tilde{\sigma}(x,s)^2$.
We shall use the simplified notations
$$b'(s(x)):=\frac{\partial b}{\partial s}(x,s(x)), 
\quad q'(s(x)):=\frac{\partial q}{\partial s}(x,s(x)) \, .$$

Let $u^*$ be an optimal control problem for $(P)$ and $p$ be the weak solution to 
\begin{equation}
\left\{ \begin{array}{ll}
\displaystyle
\frac{\partial p }{\partial t}(t,x)=-b(x,(Su^*)(x))\cdot \nabla p(t,x) \qquad & \\
\hspace{10mm} \displaystyle -\frac{1}{2}q(x,(Su^*)(x))
\Delta p(t,x)-G(t,x), &t\in (0,T), \ x\in \mathbb{R}^d ,\vspace{2mm}\\
p(T,x)=G_T(x), & x\in \mathbb{R}^d \, .
\end{array} \right. \label{dual}
\end{equation}
It follows exactly as in Theorem 5.4 in \cite{anita2} that
there exists a unique  $p\in L^2(0,T;H^2(\mathbb{R}^d))\cap C([0,T]; H^1(\mathbb{R}^d)) \cap W^{1,2}([0,T]; L^2(\mathbb{R}^d ))$ such that
$$\frac{d p }{d t}(t)=-b(\cdot ,(Su^*)(\cdot ))\cdot \nabla p(t) -\frac{1}{2}q(\cdot ,(Su^*)(\cdot ))
\Delta p(t)-G(t)$$
a.e. $t\in (0,T)$, in $L^2(\mathbb{R}^d)$ and
$$p(T,x)=G_T(x) \quad \mbox{\rm a.e. } x\in \mathbb{R}^d \, .$$
\vspace{3mm}

\begin{theorem} (Maximum Principle) 
If $u^*$ is an optimal control for (P), then
$$u^*(x)=\underset{w\in [0,M_0]}\argmin \ \Big\{ h(w)+{w}\int_0^T\int_{\mathbb{R}^d} K(x'-x)\frac{\partial b}{\partial s}\Big(x', \int_{\mathbb{R}^d}K(x'-x'')u^*(x'')dx''\Big) $$
$$\cdot \nabla p(t,x')\rho^{u^*}(t,x') \, dx' \, dt  $$
$$+\frac{w}{2}\int_0^T\int_{\mathbb{R}^d} K(x'-x)\frac{\partial q}{\partial s}\Big(x', \int_{\mathbb{R}^d}K(x'-x'')u^*(x'')dx''\Big) \rho^{u^*}(t,x')\Delta p(t,x') \, dx' \, dt \Big\} $$
a.e. $x\in \omega $.
\end{theorem}

\begin{proof}

For an arbitrary but fixed $x_0\in \omega $, there exists a positive $\varepsilon $ such that
$B_{\varepsilon }=B(x_0;\varepsilon )\subset \omega $. For an arbitrary but fixed $x_0\in \omega $ and for $\varepsilon >0$ sufficiently small
we define the following spike perturbation (\cite{barbu,sumin,anita}) of $u^*$
$$u_{\varepsilon }(x)=\left\{ \begin{array}{ll}
w, \quad & x\in B_{\varepsilon }, \\
u^*(x), \quad & x\in \omega \setminus B_{\varepsilon } ,
\end{array}
\right. $$
where $w\in [0,M_0]$ is arbitrary but fixed.
This yields
$$J(u^*)\leq J(u_{\varepsilon }), $$
which implies that
\begin{equation}\label{ineq}
0\leq \int_0^T\int_{\mathbb{R}^d}G(t,x)z_{\varepsilon }(t,x) \, dx \, dt+\int_{\mathbb{R}^d}G_T(x)z_{\varepsilon }(T,x) \, dx
+\int_{B_{\varepsilon }}[h(w)-h(u^*(x))] \, dx \, ,
\end{equation}
where $z_{\varepsilon }=\rho ^{u_{\varepsilon }}-\rho^{u^*}$. It is obvious that $z_{\varepsilon }$ is the unique weak solution 
(in $L^2(0,T;H^1(\mathbb{R}^d))\cap C([0,T]; L^2(\mathbb{R}^d))
\cap W^{1,2}(0,T; H^{-1}(\Omega ))$) to

\begin{equation*}
\left\{ \begin{array}{ll}
\displaystyle
\frac{\partial z }{\partial t}(t,x)=-\nabla \cdot [b((Su^*)(x))z(t,x)+(b((Su_{\varepsilon })(x))-b((Su^*)(x)))\rho ^{u_{\varepsilon }}(t,x)]  \vspace{2mm} \\
\displaystyle
 \hspace{10mm}+\frac{1}{2}\Delta \Big[ q((Su^*)(x)z(t,x)+[q((Su_{\varepsilon })(x))-q((Su^*)(x))]\rho ^{u_{\varepsilon }}(t,x) \Big],  \vspace{2mm} \\
~ \hspace{5cm} t\in (0,T), \ x\in \mathbb{R}^d , \\
z(0,x)=0,  \hspace{3cm} x\in \mathbb{R}^d \, .
\end{array} \right. \label{var}
\end{equation*}
By (\ref{dual}) we get that
$$0=\int_{\mathbb{R}^d}z_{\varepsilon }(T,x)G_T(x) \, dx-\int_0^T\int_{\mathbb{R}^d}z_{\varepsilon }(t,x)\frac{\partial p}{\partial t}(t,x) dx \, dt $$ 
$$-\int_0^T\int_{\mathbb{R}^d}[b((Su^*)(x)z_{\varepsilon }(t,x)+(b((Su_{\varepsilon })(x))-b((Su^*))(x))\rho ^{u_{\varepsilon }}(t,x)]\cdot \nabla p(t,x) \, dx \, dt $$
$$+\frac{1}{2}\int_0^T\int_{\mathbb{R}^d}\nabla p(t,x)\cdot \nabla \Big[ q((Su^*)(x))z_{\varepsilon }(t,x) +[q((Su_{\varepsilon })(x))-q((Su^*)(x))]\rho ^{u_{\varepsilon }}(t,x) \Big] \, dx \, dt $$
$$=-\int_0^T\int_{\mathbb{R}^d}[(b((Su_{\varepsilon })(x))-b((Su^*))(x))\rho ^{u_{\varepsilon }}(t,x)] \cdot \nabla p(t,x)dx \, dt$$
$$-\frac{1}{2}\int_0^T\int_{\mathbb{R}^d}[q((Su_{\varepsilon })(x))-q((Su^*)(x))]\rho ^{u_{\varepsilon }}(t,x) \Delta p(t,x) \, dx \, dt .$$
$$+\int_{\mathbb{R}^d}z_{\varepsilon }(T,x)G_T(x) \, dx+\int_0^T\int_{\mathbb{R}^d}G(t,x)z_{\varepsilon }(t,x) \, dx \, dt $$
and consequently
$$\int_0^T\int_{\mathbb{R}^d}G(t,x)z_{\varepsilon }(t,x) \, dx \, dt+\int_{\mathbb{R}^d}G_T(x)z_{\varepsilon }(T,x) \, dx$$
$$=\int_0^T\int_{\mathbb{R}^d}[(b((Su_{\varepsilon })(x))-b((Su^*))(x))\rho ^{u_{\varepsilon }}(t,x)] \cdot \nabla p(t,x)dx \, dt$$
$$+\frac{1}{2}\int_0^T\int_{\mathbb{R}^d}[q((Su_{\varepsilon })(x))-q((Su^*)(x))]\rho ^{u_{\varepsilon }}(t,x) \Delta p(t,x) \, dx \, dt .$$
By (\ref{ineq}) and the last computation, applying the main value theorem, we derive that
$$0\leq \int_{B_{\varepsilon }}[h(w)-h(u^*(x))] \, dx+\int_0^T\int_{\mathbb{R}^d}[(b((Su_{\varepsilon })(x))-b((Su^*))(x))\rho ^{u_{\varepsilon }}(t,x)] \cdot \nabla p(t,x)dx \, dt$$ 
$$+\frac{1}{2}\int_0^T\int_{\mathbb{R}^d}[q((Su_{\varepsilon })(x))-q((Su^*)(x))]\rho ^{u_{\varepsilon }}(t,x) \Delta p(t,x) \, dx \, dt $$
$$=\int_{B_{\varepsilon }}[h(w)-h(u^*(x))] \, dx
+\int_0^T\int_{\mathbb{R}^d}b'(\zeta _{\varepsilon}(x))((Su_{\varepsilon })(x)-(Su^*)(x))\rho ^{u_{\varepsilon }}(t,x)\cdot \nabla p(t,x)dx \, dt$$
$$+\frac{1}{2}\int_0^T\int_{\mathbb{R}^d}[q'(\xi _{\varepsilon } (x))((Su_{\varepsilon })(x)-(Su^*)(x))\rho ^{u_{\varepsilon }}(t,x) \Delta p(t,x) \, dx \, dt ,$$
where $\xi _{\varepsilon }(x), \zeta _{\varepsilon }(x)$ are convex combinations of $(Su_{\varepsilon })(x)$ and $(Su^*)(x)$.
It follows that
$$0\leq \int_{B_{\varepsilon }}[h(w)-h(u^*(x))] \, dx+\int_{B_{\varepsilon }}(w-u^*(x))S^*\Big( b'(\zeta _{\varepsilon })\cdot \int_0^T\rho ^{u_{\varepsilon }}(t)\nabla p(t) \, dt\Big)(x) \, dx$$
$$+\frac{1}{2}\int_{B_{\varepsilon }}(w-u^*(x))S^*\Big( q'(\xi _{\varepsilon })\int_0^T\rho ^{u_{\varepsilon }}(t)\Delta p(t) \, dt\Big)(x) \, dx $$
$$= \int_{B_{\varepsilon }}[h(w)-h(u^*(x))] \, dx
+\int_{B_{\varepsilon }}(w-u^*(x))S^*\Big( b'(Su^*)\cdot \int_0^T\rho^{u^*}(t)\nabla p(t) \, dt\Big)(x) \, dx $$
$$+\int_{\mathbb{R}^d}(S(u_{\varepsilon }-u^*))(x)\Big( [b'(\zeta _{\varepsilon }(x))-b'((Su^*)(x))]\cdot \int_0^T\rho ^{u_{\varepsilon }}(t,x)\nabla p(t) \, dt\Big)(x) \, dx $$
$$+\int_{\mathbb{R}^d}(S(u_{\varepsilon }-u^*))(x)\Big( b'((Su^*)(x))\cdot \int_0^T(\rho ^{u_{\varepsilon }}(t,x)-\rho^{u^*}(t,x))\nabla p(t) \, dt\Big) (x) \, dx  $$
$$ +\frac{1}{2}\int_{B_{\varepsilon }}(w-u^*(x))S^*\Big( q'(Su^*)\int_0^T\rho^{u^*}(t)\Delta p(t) \, dt\Big)(x) \, dx $$
$$+\frac{1}{2}\int_{\mathbb{R}^d}(S(u_{\varepsilon }-u^*))(x)\Big( [q'(\xi_{\varepsilon }(x))-q'((Su^*)(x))]\int_0^T\rho ^{u_{\varepsilon }}(t,x)\Delta p(t,x) \, dt\Big)(x) \, dx $$
$$+\frac{1}{2}\int_{\mathbb{R}^d}(S(u_{\varepsilon }-u^*))(x)\Big( q'((Su^*)(x))\int_0^T(\rho ^{u_{\varepsilon }}(t,x)-\rho^{u^*}(t,x))\Delta p(t,x) \, dt\Big) (x) \, dx \, ,$$
where $S^*\in L(L^2(\mathbb{R}^d);L^2(\mathbb{R}^d))$ is the adjoint of the operator $S$.

We have that
$$\frac{1}{m(B_{\varepsilon })} \int_{B_{\varepsilon }}[h(w)-h(u^*(x))] \, dx\longrightarrow h(w)-h(u^*(x_0)) $$
a.e. $x_0\in \omega $, as $\varepsilon \rightarrow 0$ and
$$\frac{1}{2m(B_{_\varepsilon })}\int_{B_{\varepsilon }}(w-u^*(x))S^*\Big( q'(Su^*)\int_0^T\rho^{u^*}(t)\Delta p(t) \, dt\Big)(x) \, dx $$
$$\longrightarrow \frac{1}{2} 
(w-u^*(x_0))S^*\Big( q'(Su^*)\int_0^T\rho^{u^*}(t)\Delta p(t) \, dt\Big)(x_0)$$
a.e. $x_0\in \omega $, as $\varepsilon \rightarrow 0$.

Note that
$$|(Su_{\varepsilon })(x)-(Su^*)(x)|=\Big| \int_{B_{\varepsilon }}K(x-x')(w-u^*(x')) \, dx'\Big| \leq M_0\| K\|_{L^{\infty }}m(B_{\varepsilon }) \, $$
for any $x\in \mathbb{R}^d$ and
$$|\xi_{\varepsilon }(x)-(Su^*)(x)|\leq  M_0\| K\|_{L^{\infty }}m(B_{\varepsilon }) \, . $$

It follows that there exists a positive constant $M_1$ such that 
$$ \frac{1}{2m(B_{\varepsilon })}\Big| \int_{\mathbb{R}^d}(S(u_{\varepsilon }-u^*))(x)\Big( [q'(\xi_{\varepsilon }(x))-q'((Su^*)(x))]\int_0^T\rho ^{u_{\varepsilon }}(t,x)\Delta p(t,x) \, dt\Big)(x) \, dx \Big| $$
$$\leq \frac{1}{2m(B_{\varepsilon })} M_1m(B_{\varepsilon })^2\int_0^T  \int_{\mathbb{R}^d} \rho ^{u_{\varepsilon }}(t,x)|\Delta p(t,x)| \, dx \, dt \longrightarrow 0, $$
as $\varepsilon \rightarrow 0$.
On the other hand, there exists a positive constant $M_2$ such that
$$\frac{1}{2m(B_{\varepsilon })}M_2m(B_{\varepsilon })\int_0^T\int_{\mathbb{R}^d}|\rho ^{u_{\varepsilon }}(t,x)-\rho^{u^*}(t,x))| \, |\Delta p(t,x)| \, dx \, dt $$
$$\leq \frac{M_2}{2}\| \rho ^{u_{\varepsilon }}-\rho^{u^*}\|_{L^2}\| \Delta p\| _{L^2}\longrightarrow 0 \, $$
as $\varepsilon \rightarrow 0$.

On the other hand
$$\frac{1}{m(B_{_\varepsilon })}\int_{B_{\varepsilon }}(w-u^*(x))S^*\Big( b'(Su^*)\cdot \int_0^T\rho^{u^*}(t)\nabla p(t) \, dt\Big)(x) \, dx $$
$$\longrightarrow 
(w-u^*(x_0))S^*\Big( b'(Su^*)\cdot \int_0^T\rho^{u^*}(t)\nabla p(t) \, dt\Big)(x_0)$$
a.e. $x_0\in \omega $, as $\varepsilon \rightarrow 0$.

Note that
$$|\zeta _{\varepsilon }(x)-(Su^*)(x)|\leq  M_0\| K\|_{L^{\infty }}m(B_{\varepsilon }) \, . $$

It follows that there exists a positive constant $M_2$ such that 
$$ \frac{1}{m(B_{\varepsilon })}\Big| \int_{\mathbb{R}^d}(S(u_{\varepsilon }-u^*))(x)\Big( [b'(\zeta _{\varepsilon }(x))-b'((Su^*)(x))]\cdot \int_0^T\rho ^{u_{\varepsilon }}(t,x)\nabla p(t,x) \, dt\Big)(x) \, dx \Big| $$
$$\leq \frac{1}{m(B_{\varepsilon })} M_2m(B_{\varepsilon })^2\int_0^T  \int_{\mathbb{R}^d} \rho ^{u_{\varepsilon }}(t,x)|\nabla p(t,x)| \, dx \, dt \longrightarrow 0, $$
as $\varepsilon \rightarrow 0$.
On the other hand, there exists a positive constant $M_3$ such that
$$\frac{1}{m(B_{\varepsilon })}M_3m(B_{\varepsilon })\int_0^T\int_{\mathbb{R}^d}|\rho ^{u_{\varepsilon }}(t,x)-\rho^{u^*}(t,x))| \, |\nabla p(t,x)| \, dx \, dt $$
$$\leq {M_3}\| \rho ^{u_{\varepsilon }}-\rho^{u^*}\|_{L^2}\| \nabla p\| _{L^2}\longrightarrow 0 \, $$
as $\varepsilon \rightarrow 0$.

If we take into account the last convergences, we get that
$$0\leq h(w)-h(u^*(x_0))$$
$$ +(w-u^*(x_0))S^*\Big( b'(Su^*)\cdot \int_0^T\rho^{u^*}(t)\nabla p(t) \, dt\Big)(x_0),$$
$$+\frac{1}{2}(w-u^*(x_0))S^*\Big( q'(Su^*) \int_0^T\rho^{u^*}(t)\Delta p(t) \, dt\Big)(x_0),$$
a.e. $x_0\in \omega $.

This implies (Maximum Principle) 
$$u^*(x)=\argmin_{w\in [0,M_0]}\Big\{ h(w)+{w}S^*\Big( b'(Su^*)\cdot \int_0^T\rho^{u^*}(t)\nabla p(t) \, dt\Big)(x)\Big\} $$
$$+\frac{w}{2}S^*\Big( q'(Su^*)\int_0^T\rho^{u^*}(t)\Delta p(t) \, dt\Big)(x)\Big\} $$
a.e. $x\in \omega $ and so we get the conclusion of the theorem.
\end{proof}

\section{Mathematical Finance application}
In what follows, we provide a financial application rooted in stochastic optimal control and leveraging the general case where \(\sigma = \sigma(t,x,s)\) is matrix-valued. In particular, we develop a robust framework for an optimal asset allocation problem in a mean-variance portfolio setting with diffusion control. This setup closely aligns with advanced financial contexts, where both drift and volatility terms are influenced by the control variable, specifically through an indirect mechanism of convolution.
Accordingly, we consider an investor who aims to maximize a utility-adjusted wealth objective over a finite time horizon \( [0, T] \) by optimally allocating funds across risky assets. In this context, the investor's control variable, \( u(x) \), influences both the drift \( b \) and volatility \(\sigma\). In particular, let the controlled wealth process \( X_t \) follows the SDE:
\[
dX_t = b(t, X_t, (Su)(X_t)) \, dt + \sigma(t, X_t, (Su)(X_t)) \, dW_t,
\]
where \( (W_t)_{t\in [0,T]} \) is a \(N\)-dimensional Brownian motion, while
\( u \in L^2(\omega) \) is the control variable influencing the portfolio’s drift and volatility and \((Su)(x)\) represents a nonlocal action of the control \( u \) through a convolution operation:
\[
(Su)(x) = \int_{\mathbb{R}^d} K(x - x') u(x') \, dx',
\]
where \( K(x - x') \) is a kernel satisfying regularity conditions mentioned in the theoretical part previously developed.
Then, the investor's objective is to minimize a mean-variance-based cost functional over the investment horizon:
\[
J(u) = \mathbb{E} \left[ \int_0^T G(t, X_t) \, dt + G_T(X_T) \right] + \int_{\omega} h(u(x)) \, dx,
\]
where:
 \( G(t, x) \) and \( G_T(x) \) represent running and terminal cost terms associated with the portfolio’s wealth and \( h(u(x)) \) is a convex penalization function on the control, ensuring regularization and feasibility.

To determine the existence and uniqueness of the optimal control \( u^(x) \), we employ the theoretical results on Fokker-Planck equations with nonlocal actions (see above). 
The associated optimal control problem thus translates to finding a control \( u \in {\cal U} \) such that \( J(u) \) is minimized subject to the controlled dynamics.

Under the assumptions in section 2 there exists at least one optimal control $u^*$.

The provided framework applies as follows: \( b \) represents the expected return rate influenced by the control \( u(x) \), modulating the investor’s exposure to market factors, while the diffusion matrix \( \sigma \) models the asset return volatility, directly controlled by \( u(x) \), allowing the investor to manage risk dynamically by adjusting \( \sigma \) over time and space and the penalization \( h(u) \) reflects costs or risks associated with high exposure levels, encouraging controls that optimize risk-adjusted returns.
Hence providing a mathematically rigorous basis for mean-variance optimization where both drift and volatility controls are nonlocally influenced, enabling an asset allocation strategy within a controlled stochastic environment.

To better specify, let us underline
that in a financial market, an investor might need to control their portfolio's exposure to certain market factors (e.g., sector-specific volatilities or macroeconomic variables) that do not act independently but rather have a nonlocal or interconnected impact across assets. Here, the control \( u(x) \) affects both the drift and the diffusion matrix through convolution:
   \[
   (Su)(x) = \int_{\mathbb{R}^d} K(x - x') u(x') \, dx',
   \]
   where \( K(x - x') \) captures the spatial dependencies across different components of the portfolio, e.g., \( K(x - x') \) could represent sectorial or asset correlations, reflecting how the control \( u \) in one area (e.g., one asset class or market sector) affects other areas. 
Moreover, the Fokker-Planck equation, as modified by the control, describes the evolution of the portfolio’s probability distribution, \(\rho(t, x)\), over time, and is sensitive to how \(\sigma\) varies with the control.
Hence, by leveraging provided results to the controlled Fokker-Planck equation (\ref{FP1}), 
where \( \sigma(t,x, (Su)(x)) \) is matrix-valued and depends on the control \( u \) in a nontrivial manner, the investor can predict how changes in the control \( u(x') \), applied spatially across the asset spectrum, will influence the entire portfolio distribution over time—a fundamental requirement for coherent risk management in mean-variance frameworks.

Finally, the stated maximum principle result provided is particularly valuable in portfolio optimization because it enables us to characterize the optimal control \( u^*(x) \) in terms of an integral representation (under appropriate hypotheses):
   \begin{equation*} 
\begin{split}
   & u^*(x) = \argmin_{w \in [0, M_0]} \Big\{ h(w) + w \int_0^T \int_{\mathbb{R}^d} K(x' - x) b'(Su^*)(x')   \\
 & \hspace{3cm} \cdot  \nabla p(t, x')\rho^{u^*}(t, x') \, dx' \, dt   \\ 
   & + \frac{w}{2} \int_0^T \int_{\mathbb{R}^d} K(x' - x) q'(Su^*)(x') \rho^{u^*}(t, x') \Delta p(t, x') \, dx' \, dt \Big\}
  \end{split}
\end{equation*}
which shows that \( u^*(x) \), the optimal spatial control, depends on both the drift and diffusion terms and incorporates the distributional state variables \( \rho^{u^*}(t, x') \) and \( \nabla p(t, x') \), effectively linking the current portfolio strategy to expected future distributions. In financial terms, this optimality condition enables investors to anticipate and adjust their controls based on future distributions of returns and risks, taking advantage of the nonlocal influence of each asset.

It is also worth mentioning that the nonlocal control effect encoded by \( (S u)(x) \) directly leads to an integral representation of the penalization term \( h(u(x)) \), which captures costs related to maintaining a desired risk level across the portfolio:
   \[
   J(u) = \int_0^T \mathbb{E} \left[ G(t, X_t) \right] dt + \mathbb{E}[G_T(X_T)] + \int_{\omega} h(u(x)) dx.
   \]
and since \( h \) is convex, this integral representation ensures that the investor minimizes the combined expected portfolio volatility and tracking error costs associated with implementing the control, \( u(x) \), over time.

\section{Final comments and further directions}

First of all, we note that the results in section 3 can be extended to the case when the following hypotesis holds instead of (H5):

\begin{itemize}
\item[\bf (H5')] $b$, $\sigma $ are time-independent, $b\in C^{0,2}_b(\mathbb{R}^d\times [0,+\infty);\mathbb{R}^d)$ and
$$G_T\in H^1(\mathbb{R}^d) \, .$$
\end{itemize}
Actually, the following result holds:
\vspace{2mm}

\it
If $p$ is the weak solution to 
\begin{equation*}
\left\{ \begin{array}{ll}
\displaystyle
\frac{\partial p }{\partial t}(t,x)=-b(x,(Su^*)(x))\cdot \nabla p(t,x) \qquad & \\
\hspace{10mm} \displaystyle -\frac{1}{2}q_{ij}(x,(Su^*)(x))
\frac{\partial ^2 p}{\partial x_i \partial x_j}(t,x)-G(t,x), &t\in (0,T), \ x\in \mathbb{R}^d ,\vspace{2mm}\\
p(T,x)=G_T(x), & x\in \mathbb{R}^d \, ,
\end{array} \right. \label{dual}
\end{equation*}
where $u^*$ is an optimal control for $(P)$, then 
$$u^*(x)=\underset{w\in [0,M_0]}\argmin \ \Big\{ h(w)+{w}\int_0^T\int_{\mathbb{R}^d} K(x'-x)\frac{\partial b}{\partial s}\Big(x', \int_{\mathbb{R}^d}K(x'-x'')u^*(x'')dx''\Big) $$
$$\cdot \nabla p(t,x')\rho^{u^*}(t,x') \, dx' \, dt  $$
$$+\frac{w}{2}\int_0^T\int_{\mathbb{R}^d} K(x'-x)\frac{\partial q_{ij}}{\partial s}\Big(x', \int_{\mathbb{R}^d}K(x'-x'')u^*(x'')dx''\Big) \rho^{u^*}(t,x')\frac{\partial ^2 p}{\partial x_i \partial x_j}(t,x') \, dx' \, dt \Big\} $$
a.e. $x\in \omega $. \rm
\vspace{2mm}

We chose to present the result in the form seen in section 3 in order to simplify the notation, make it easier to follow the proof and emphasize the main ideas. 

\vspace{8mm}

If one considers the particular drift and diffusion coefficients
$$b(t,x,u)=b_0(t,x)u, \quad \sigma (t,x,u)=\sigma _0(t,x)u \, ,$$
and appropriate assumptions on $b_0, \sigma _0$, 
and $Su=u$, ${\cal U}=\{ v\in L^{\infty }(\omega ) \ : 0\leq u(x)\leq M_0 \ \mbox{\rm a.e. } x\in \omega \}$,
then we expect to get the existence of the optimal control. 
Here ${\cal U}$ is not a compact subset of $L^2(\omega )$ but 
for any sequence $\{ u_n\}_{n\in \mathbb{N}^*} \in {\cal U}$, $u_n \longrightarrow u$ in the weak $*$ topology of $L ^{\infty}(\omega)$ we expect to get via Aubin's compactness theorem that on a subsequence 
(also indexed by $n$) we have $\rho ^{u_n}\longrightarrow \rho $ in $L^2(0,T;L _{loc}^2(\mathbb{R}^d))$  and that $\rho \equiv \rho ^u$ 
(the linearity of $b$ and $\sigma $ with respect to $u$ is crucial for that). An in-depth argument will allow us to obtain the convergence $\rho ^{u_n}\longrightarrow \rho ^u $ in $L^2(0,T;L ^2(\mathbb{R}^d))$. 
For specific control problems, these two convergences are sufficient to conclude the existence of an optimal control.  
The maximum principle follows as in section 3 and has a simpler form. 
\vspace{5mm}

It is of special interest the situation (see \cite{anita2,anita3}) when
$Su \in H^1(\mathbb{R}^d)$ is the unique weak solution to the following equation
\begin{equation}\label{lax}
\eta (x)-\Delta \eta (x)=u(x), \quad x\in \mathbb{R}^d \, ,
\end{equation}
i.e. $\eta \in H^1(\mathbb{R}^d)$ such that $\langle \eta ,\varphi \rangle _{H^1}=\langle u ,\varphi \rangle _{L^2}$, for any $\varphi \in H^1(\mathbb{R}^d)$.
The existence and uniqueness of a unique weak solution to (\ref{lax}) follows via Lax-Milgram's lemma. 
This situation corresponds to another control with nonlocal action.

The term $Su $ represents the nonlocal action of the control $u $. In the context of population dynamics
$\eta (x)$ represents the density per unit of time of an infusion (constant-in-time) of a substance which improves/ stimulates or diminishes the movement performance of the
given population of density $\rho $ or of the properties of the environment with respect to the movement. This substance diffuses, with a diffusion coefficient equal to $1$, and decays at a $1$ rate due to different causes. Denote the density 
of this substance at time $\tau$ and location $x$ by $\tilde{\eta }(\tau,x)$; $\tilde{\eta }\in L^2_{loc}([0,+\infty ); V)\cap W^{1,2}_{loc}([0,+\infty );H^{-1}(\mathbb{R}^d))$ is the weak solution to
$$\left\{ \begin{array}{ll}
\displaystyle
\frac{\partial \tilde{\eta }}{\partial t}(\tau ,x)=-\tilde{\eta }(\tau ,x)+\Delta \tilde{\eta }(\tau ,x)+u (x), \quad & x\in \mathbb{R}^d, \ \tau\in [0,+\infty ),  \vspace{2mm} \\
\tilde{\eta}(0,x)=0, & x\in \mathbb{R}^d .
\end{array}
\right. $$ 
For any $u\in {\cal U}$ we have that
$$\tilde{\eta }(\tau ,\cdot )-\eta \longrightarrow 0 \quad \mbox{\rm in } L^2(\mathbb{R}^d), \ \mbox{\rm as } \tau \rightarrow +\infty .$$
This convergence is as fast as $\exp (-\tau )$; therefore, after a sufficiently large time $\tau $, $\eta $ represents a good approximation of $\tilde{\eta }(\tau ,\cdot )$, hence we expect that a similar approach could be used.
\vspace{5mm} 

Another possible extension of our investigation regards the optimal control problems related to nonlinear or generalized Fokker-Planck equations with controls acting on the diffusion
(see e.g. \cite{dipersio}). A promising possibility is to work in the space $H^{-1}(\mathbb{R}^d)$ and to use a semigroup approach (as in \cite{anita3}).


\end{document}